\documentclass[12pt]{article}
\usepackage{amsfonts, amsmath, amssymb, latexsym, eucal, amscd} 

\usepackage{amsthm}
\usepackage{amscd}

\usepackage{cite}

\newtheorem{definition}{Definition}[section]
\newtheorem{theorem}[definition]{Theorem}
\newtheorem{lemma}[definition]{Lemma}
\newtheorem{corollary}[definition]{Corollary}

\newtheorem{note}[definition]{Note}

\newtheorem{proposition}[definition]{Proposition}

\begin{document} 

\title{\bf A $Q$-polynomial structure associated \\with the projective geometry $L_N(q)$
}
\author{
Paul Terwilliger}
\date{}

\maketitle
\begin{abstract} 
There is a type of distance-regular graph, said to be $Q$-polynomial. In this paper
we investigate a generalized $Q$-polynomial property involving a graph that is not necessarily distance-regular. We give a detailed
description of an example associated with the projective geometry $L_N(q)$.

\bigskip

 \end{abstract}
 
 \section{Introduction}
 There is a type of finite, undirected, connected graph, said to be distance-regular \cite{bannai, bcn, dkt, bbit}. These graphs have 
 the sort of combinatorial regularity that can be analyzed using algebraic methods, such as 
 linear algebra (eigenvalues/eigenvectors of the adjacency matrix \cite[Section~4.1]{bcn}, tridiagonal pairs \cite{itt}); geometry (linear programming bounds \cite{delsarte}, root systems 
 \cite[Chapter~3]{bcn});
  special functions (orthogonal polynomials \cite{leonard, LSnotes}, hypergeometric series \cite[Chapter 3]{bannai});
   and representation theory (the subconstituent algebra \cite{terwSub1, terwSub2, terwSub3, int}, the $q$-Onsager algebra \cite{itoOq, augIto}). 
   \medskip
   
   \noindent  For an integer $N\geq 1$, the $N$-cube $Q_N$ is an attractive example of a distance-regular graph. To define $Q_N$, start with a set $\mathbb S$ of cardinality $N$.
   The vertex set of $Q_N$ consists of the subsets of $\mathbb S$.  Vertices $y,z $ of $Q_N$ are adjacent whenever
   one of $y,z$ contains the other, and their cardinalities differ by one.
   For each vertex of $Q_N$ the corresponding subconstituent algebra was described by Go \cite{go}.
   \medskip
   
   \noindent 
   There is a type of distance-regular graph, said to be $Q$-polynomial  \cite[Chapter~8]{bcn}.
The $Q$-polynomial property was introduced by Delsarte \cite{delsarte} in his study of coding theory, and has been investigated intensely  ever since
 \cite{bannai, bbit, bcn, dkt, go,terwSub1,terwSub2, terwSub3, int}.
The $N$-cube is $Q$-polynomial; see \cite{delsarte} or \cite[Section~12]{go}.
\medskip

\noindent  We emphasize one feature of a $Q$-polynomial distance-regular graph $\Gamma$.
 For each vertex $x$ of $\Gamma$ there exists a certain diagonal matrix $A^*=A^*(x)$, called the dual adjacency matrix of $\Gamma$ with respect to $x$ \cite[Section~2]{pasc}. 
  The eigenspaces of $A^*$  are the
  subconstituents of $\Gamma$ with respect to $x$. The adjacency matrix $A$ of $\Gamma$ is related to $A^*$ by the fact that
  each of $A, A^*$ acts on the eigenspaces of the other one in a (block) tridiagonal fashion \cite[Section~13]{int}.   %%%This feature motived the concept of a tridiagonal pair \cite{itt}.
 \medskip
 
 \noindent 
 Guided by the dual adjacency matrix concept, in \cite[Section~20]{int} we generalized the $Q$-polynomial property in three directions:
 (i) we drop the assumption that $\Gamma$  is distance-regular; 
  (ii) we drop the assumption that every vertex of $\Gamma$ has a dual adjacency matrix, and instead require that one distinguished vertex of $\Gamma$ has a dual adjacency matrix;
 (iii) we replace the adjacency matrix of $\Gamma$ by a weighted adjacency
 matrix. 
 The generalized $Q$-polynomial property is described in Definitions \ref{def:dualA}--\ref{def:Qx}
 below.
 \medskip
 
 \noindent Broadly speaking, our goal in this paper is to convince the reader that the generalized $Q$-polynomial property is worth investigating. To this end,
 we will give a detailed description of 
 one attractive example, associated with the projective geometry $L_N(q)$. As we will see, it is natural to view $L_N(q)$ as a $q$-analog of the $N$-cube $Q_N$.
 \medskip
 
 \noindent 
  Given a finite field  ${\rm GF}(q)$ and an integer $N\geq 1$, the graph $L_N(q)$ is defined as follows.
  Start with a vector space $\mathbb V$ over ${\rm GF}(q)$  that has dimension $N$. The vertex set of $L_N(q)$ consists of the subspaces of $\mathbb V$.
  Vertices $y, z$ of $L_N(q)$ are adjacent whenever one of $y,z$ contains the other, and their dimensions differ by one. The graph $L_N(q)$ is not distance-regular
  if $N\geq 2$.  \medskip
  
  \noindent
  Background information about $L_N(q)$ can be found in 
  \cite[Section~9.3]{bcn},\cite[Chapter~1]{cameron}, \cite{murali}, \cite{murali2}, \cite[Example~3.1(5) with $M=N$]{uniform},  \cite[Section~7]{introLP}.
  \medskip

  \noindent Let $\bf 0$ denote the zero subspace of $\mathbb V$. We distinquish the vertex $\bf 0$ of $L_N(q)$.
  We will work with the weighted adjacency matrix $A$ of $L_N(q)$ introduced by S. Ghosh and M. Srinivasan in their excellent paper
\cite{murali}; see Definition 5.1 below.
   We define a diagonal matrix $A^*$ such that for each vertex $y$ of $L_N(q)$, the
  $(y,y)$-entry of $A^*$ is $q^{-{\rm dim}\,y}$. We remark that $A^*$ is a scalar multiple of the matrix $K$ that appears in
 \cite[Section~7]{introLP}. By construction, the eigenspaces of $A^*$ are the subconstituents of $L_N(q)$ with respect to $\bf 0$.
  By \cite[Section~1]{murali} the matrix $A$ is diagonalizable over $\mathbb R$.
  The eigenvalues and eigenspaces of $A$ are computed in \cite[Section~2]{murali}; see Lemma \ref{lem:Mur} below. 
  %%%\label{lem:Mur} {\rm (See \cite[Section~2]{murali}.)} 
 We will show that $A^*$ acts on the eigenspaces of $A$ in a (block) tridiagonal fashion. This will imply that $A$ is $Q$-polynomial
  with respect to $\bf 0$, in the sense of Definition \ref{def:Qx} below. This fact is the main result of the paper.
  \medskip
  
  \noindent We just described our main result. We also obtain the following auxiliary results, which may be of independent interest.
  We display four bases for the standard module of $L_N(q)$, said to be split.
  With respect to each split basis, one of $A, A^*$ acts in an upper triangular fashion, and the other acts in a lower triangular fashion.
  We compute the actions of $A$ and $A^*$ on each split basis.
  For the vertex $\bf 0$  the associated subconstituent algebra $T$ is generated by $A, A^*$.
  We show that $A, A^*$ satisfy the tridiagonal relations. We show that the standard module of $L_N(q)$ decomposes into a direct sum of irreducible $T$-modules. 
  We describe the irreducible $T$-modules in terms of Leonard systems of dual $q$-Krawtchouk type.
  \medskip

 \noindent We mentioned above that the weighted adjacency matrix $A$ of $L_N(q)$ was introduced in \cite{murali}. We would like
 to acknowledge that an affine transformation $\alpha A+\beta I$ appears in the work of 
 Bernard, Cramp{\'e}, Vinet \cite[Theorem~7.1]{PAB} concerning the subconstituent algebra of the symplectic dual polar graph. The spectrum of $A$
 can be deduced from \cite[Section~7.2]{PAB}, and Leonard systems of dual $q$-Krawtchouk type are alluded to in \cite[line (129)]{PAB}.
 \medskip

 \noindent This paper is organized as follows. 
 In Section 2 we discuss the generalized $Q$-polynomial property. 
 In Section 3 we define the projective geometry $L_N(q)$, and discuss its basic properties.
 In Section 4 we introduce the four split bases for the standard module of $L_N(q)$.
 In Section 5 we define the matrices $A$, $A^*$ and discuss their basic properties.
 In Sections 6, 7 we compute the actions of $A$ and $A^*$ on the split bases, and consider the implications.
 In Section 8 we show that the matrix $A$ is $Q$-polynomial with respect to $\bf 0$. We also show that $A, A^*$ satisfy the
 tridiagonal relations, and we describe the irreducible $T$-modules.

 \section{The $Q$-polynomial property}
 
 For distance-regular graphs the $Q$-polynomial property is well known \cite{bannai, bbit, bcn, dkt, go,terwSub1,terwSub2, terwSub3, int}.
 In \cite[Section~20]{int} we generalize this $Q$-polynomial property to a setting that involves  a weighted adjacency matrix of a graph that is not necessarily distance-regular. In the present section we review the generalized $Q$-polynomial property. 
 Our review involves  the adjacency algebra, the dual adjacency algebra,  and the subconstituent
algebra. 
%For background information we refer the reader to 
%\cite{bannai, bbit, bcn, dkt, go,terwSub1,terwSub2, terwSub3, int}. %The hypercubes are a good example of the type of
%graphs discussed in this paper. For a detailed study of the hypercubes, see \cite{go}.

\medskip
\noindent 
Let $\mathbb R$ denote the field of real numbers. 
Let $X$ denote a nonempty  finite  set.
Let ${\rm Mat}_X(\mathbb R)$
denote the $\mathbb R$-algebra
consisting of the matrices with rows and columns  indexed by $X$
and all entries in $\mathbb R  $. Let $I \in{ \rm Mat}_X(\mathbb R)$ denote the identity matrix.
 Let
$V=\mathbb R^X$ denote the vector space over $\mathbb R$
consisting of the column vectors with
coordinates indexed by $X$ and all entries 
in $\mathbb R$.
The algebra ${\rm Mat}_X(\mathbb R)$
acts on $V$ by left multiplication.
We call $V$ the {\it standard module}.
%We endow $V$ with a bilinear form $\langle \, , \, \rangle$ 
%that satisfies
%$\langle u,v \rangle = u^t v$ for 
%$u,v \in V$,
%where $t$ denotes transpose. This bilinear form is symmetric. For $u \in V$ we abbreviate $\Vert u \Vert^2 = \langle u, u \rangle $.  Note that $\Vert u \Vert^2 \geq 0$, with equality if and only if $u=0$.
%For $u, v \in V$ and $B \in {\rm Mat}_X(\mathbb R)$ we have $\langle Bu, v \rangle = \langle u, B^tv\rangle$.
For all $y \in X,$ define a vector $\hat{y} \in V$ that has $y$-coordinate  $1$ and all other coordinates $0$.
The vectors $\lbrace \hat y \rbrace_{y \in X}$ form a basis for $V$. 
%For later use, define a matrix $J \in {\rm Mat}_X(\mathbb R)$ that has all entries $1$.
\medskip

\noindent
Let $\Gamma = (X, \mathcal R)$ denote a finite, undirected, connected graph,
without loops or multiple edges, with vertex set $X$,
edge set
$\mathcal R$, and path-length distance function $\partial$.
 Vertices $y,z \in X$ are said to be {\it adjacent} whenever they form an edge. For $y \in X$ and an integer $i\geq 0$ define
the set  $\Gamma_i(y) = \lbrace z \in X \vert \partial(y,z)=i\rbrace$. We abbreviate $\Gamma(y)= \Gamma_1(y)$.
%Let $\partial $ denote the
%path-length distance function for $\Gamma $. 
      %For $x \in X$ define $D(x) = {\rm max} \lbrace \partial(x,y) \vert y \in X\rbrace$. We call $D(x)$
      %the {\it diameter of $\Gamma$ with respect to $x$}.
% Define $\Delta={\rm max} \lbrace \partial(y,z) \vert y,z \in X\rbrace$.
% We call $\Delta$ the {\it diameter of $ \Gamma$}.
The graph $\Gamma$ is said to be {\it bipartite} whenever there exists a partition $X= X^+ \cup X^{-}$ such that 
$\Gamma(y) \subseteq X^- $  for all $y \in X^+$, and
$\Gamma(y) \subseteq X^+ $ for all $y \in X^-$.
\medskip

\begin{definition} \label{def:wa} \rm By a {\it weighted adjacency matrix of $\Gamma$}, we mean a matrix
 $A \in {\rm Mat}_X(\mathbb R)$ that has 
$(y,z)$-entry
\begin{align*}
A_{y,z} = \begin{cases}  
\not=0, & {\mbox{\rm if $y, z$ are adjacent}};\\
0, & {\mbox{\rm if $y, z$ are not adjacent}}
\end{cases}
 \qquad (y,z \in X).
\end{align*}
\end{definition}
\noindent %%A weighted adjacency matrix might not be diagonalizable over $\mathbb R$ (or even $\mathbb C$).
 For the rest of this section, we fix a weighted adjacency matrix $A$ of $\Gamma$ 
that is diagonalizable over $\mathbb R$.
\medskip

\noindent Next we discuss the adjacency algebra of $\Gamma$.
Let $M$ denote the subalgebra of ${\rm Mat}_X(\mathbb R)$ generated by $A$. We call $M$
 the {\it adjacency algebra of $\Gamma$ generated by $A$}.
 Let $\mathcal D+1$ denote the dimension of the vector space $M$.
%We have $\mathcal D \geq \Delta$ because the matrices $\lbrace A^i \rbrace_{i=0}^\Delta$ are linearly independent.
%%%It turns out that $A$ generates $M$ \cite[p.~190]{bannai}.
Since $A$ is diagonalizable, the vector space $M$ has a  basis 
$\lbrace E_i\rbrace_{i=0}^{\mathcal D}$ such that  $\sum_{i=0}^{\mathcal D} E_i=I$
%%(iii) $\overline{E_i} = E_i \;(0 \le i \le D)$; 
and $E_iE_j =\delta_{i,j}E_i $ for $0 \leq i,j \leq {\mathcal D}$.
We call $\lbrace E_i\rbrace_{i=0}^{\mathcal D}$  the {\it primitive idempotents
of $A$}. 
Since $A \in M$, there exist real numbers $\lbrace \theta_i \rbrace_{i=0}^{\mathcal D}$ such that 
$
A = \sum_{i=0}^{\mathcal D} \theta_i E_i$.
The scalars $\lbrace \theta_i \rbrace_{i=0}^{\mathcal D}$ are mutually distinct since $A$ generates $M$. 
We have $AE_i = \theta_i E_i = E_i A$ for $0 \leq i \leq {\mathcal D}$. Note that
\begin{align*}
V = \sum_{i=0}^{\mathcal D} E_iV \qquad \qquad (\mbox{\rm direct sum}).
\end{align*}
For $0 \leq i \leq {\mathcal D}$ the subspace $E_iV$ is an eigenspace of $A$, and  $\theta_i$ is the corresponding eigenvalue.
For notational convenience, define $E_{-1}=0$ and $E_{\mathcal D+1}=0$.
 %%Shortly we will discuss the corresponding eigenvalues.
%%Let $m_i$ denote the dimension of $E_iV$. Note that $m_i = {\rm tr}(E_i)$, where tr denotes trace.
\medskip

\noindent Next we discuss the dual adjacency algebras of $\Gamma$. 
For the rest of this section, fix
a vertex $x \in X$. Define the integer $D=D(x)$ by
\begin{align*}
D = {\rm max}\lbrace \partial(x,y)\vert y \in X\rbrace.
\end{align*}
\noindent We call $D$ the {\it diameter of $\Gamma$ with respect to $x$}. We have $D\leq \mathcal D$, because
the matrices $\lbrace A^i \rbrace_{i=0}^D$ are linearly independent.
For 
$ 0 \leq i \leq D$ we define a diagonal matrix $E_i^*=E_i^*(x)$  in 
 ${\rm Mat}_X(\mathbb R)$
 that has $(y,y)$-entry
\begin{equation}\label{DEFDEI}
(E_i^*)_{y,y} = \begin{cases} 1, & \mbox{\rm if $\partial(x,y)=i$};\\
0, & \mbox{\rm if $\partial(x,y) \ne i$}
\end{cases}
 \qquad (y \in X).
\end{equation}
We call $\lbrace E_i^*\rbrace_{i=0}^D$ the {\it dual primitive idempotents of $\Gamma$
 with respect to $x$} \cite[p.~378]{terwSub1}. 
We have $\sum_{i=0}^D E_i^*=I$ and  $E_i^*E_j^* = \delta_{i,j}E_i^* $ for $0 \leq i,j\leq D$.
Consequently the matrices
$\lbrace E_i^*\rbrace_{i=0}^{D}$ form a 
basis for a commutative subalgebra
$M^*=M^*(x)$ of 
${\rm Mat}_X(\mathbb R)$.
We call 
$M^*$ the {\it dual adjacency  algebra of
$\Gamma$ with respect to $x$} \cite[p.~378]{terwSub1}.
\medskip

%%%It turns out that $A$ generates $M$ \cite[p.~190]{bannai}.

%\noindent
%We  recall the eigenvalues
%of  $\Gamma $.
%Since $\lbrace E_i\rbrace_{i=0}^D$ form a basis for  
%$M$ there exist real numbers  
%$\lbrace\theta_i\rbrace_{i=0}^D$
%such that
%$A = \sum_{i=0}^D \theta_iE_i$.
%Observe
%$AE_i = E_iA =  \theta_iE_i$ for $0 \leq i \leq D$.
%Observe
%$\lbrace\theta_i\rbrace_{i=0}^D$ are mutually distinct 
%since $A$ generates $M$. 
%We call $\theta_i$  the {\it eigenvalue}
%of $\Gamma$ associated with $E_i$ $(0 \leq i \leq D)$. Shortly we will see that $\lbrace \theta_i\rbrace_{i=0}^D$ are mutually distinct.
%\medskip

\noindent 
Next we recall the subconstituents of $\Gamma $ with respect to $x$.
From
\eqref{DEFDEI} we obtain
\begin{equation}\label{DEIV}
E_i^*V = \mbox{\rm Span}\lbrace \hat{y} \vert y \in \Gamma_i(x)\rbrace
\qquad \qquad (0 \leq i \leq D).
\end{equation}
By 
\eqref{DEIV}  and since $\lbrace {\hat y} \rbrace_{y \in X}$
is an  basis for $V$,
 we find
\begin{align*}
V = \sum_{i=0}^D E^*_iV \qquad \qquad (\mbox{\rm direct sum}). %\label{eq:dualEig}
\end{align*}
For $0 \leq i \leq D$ the subspace $E^*_iV$ is a common eigenspace for $M^*$. %%Note that $E^*_0 V = \mathbb R {\hat x}$.
We call $E_i^*V$ the $i^{th}$ {\it subconstituent of $\Gamma$
with respect to $x$}. %%% Note that $E^*_0 V = \mathbb R {\hat x}$. Also note that
For notational convenience, define  $E^*_{-1}=0$ and $E^*_{D+1}=0$. 
By the triangle inequality, for adjacent $y,z \in X$ the distances $\partial(x,y)$ and $\partial(x,z) $ differ by at most one.
Consequently 
\begin{align}
A E^*_iV \subseteq E^*_{i-1}V + E^*_iV + E^*_{i+1}V \qquad \qquad (0 \leq i \leq D).
\label{eq:IntroAact}
\end{align}

\noindent Next we discuss the $Q$-polynomial property.

\begin{definition}\label{def:dualA} {\rm (See \cite[Definition~20.6]{int}.)} \rm A matrix $A^* \in {\rm Mat}_X(\mathbb R)$ is called a  {\it dual adjacency matrix of $\Gamma$}  (with respect to $x$ and the ordering $\lbrace E_i \rbrace_{i=0}^{\mathcal D}$)
whenever $A^*$ generates $M^*$ and
\begin{align}
A^* E_iV \subseteq E_{i-1}V+ E_iV+E_{i+1}V \qquad \qquad (0 \leq i \leq \mathcal D).
\label{eq:IntroAsact}
\end{align}
\end{definition}

\begin{definition}\label{def:Qpoly}  {\rm (See \cite[Definition~20.7]{int}.)} \rm We say that the ordering $\lbrace E_i \rbrace_{i=0}^{\mathcal D}$ is {\it $Q$-polynomial with respect to $x$}
whenever there exists a dual adjacency matrix of $\Gamma$ with respect to $x$ and $\lbrace E_i \rbrace_{i=0}^{\mathcal D}$.
\end{definition}

\begin{definition}\rm \label{def:Qx}  {\rm (See \cite[Definition~20.8]{int}.)} \rm We say that $A$ is {\it $Q$-polynomial with respect to $x$} whenever there exists an ordering of the primitive idempotents of $A$
that is $Q$-polynomial with respect to $x$.
\end{definition}

\noindent Assume that $\Gamma$ has a dual adjacency matrix $A^*$ with respect to $x$ and $\lbrace E_i \rbrace_{i=0}^\mathcal D$. Since $A^* \in M^*$, there
exist real numbers $\lbrace \theta^*_i\rbrace_{i=0}^D$ such that 
 $A^* = \sum_{i=0}^D \theta^*_i E^*_i$.
The scalars $\lbrace \theta^*_i \rbrace_{i=0}^D$ are mutually distinct since $A^*$ generates $M^*$. We have
$A^* E^*_i = \theta^*_i E^*_i = E^*_i A^* $ for $0 \leq i \leq D$. We mentioned earlier that the sum
$V = \sum_{i=0}^D E^*_iV$ is direct.
For $0 \leq i \leq D$ the subspace $E^*_iV$ is an eigenspace of $A^*$,
and $\theta^*_i$ is the corresponding eigenvalue.
\medskip

\noindent  As we investigate the $Q$-polynomial property, it is helpful to bring in the subconstituent algebra \cite{terwSub1,
terwSub2,
terwSub3}. The following
definition is a variation on  
  \cite[Definition 3.3]{terwSub1}.

\begin{definition} \label{def:Talg} \rm 
Let $T=T(x, A)$ denote the subalgebra of ${\rm Mat}_X(\mathbb R)$ generated by 
$M$ and $M^*$.  We call $T$ the {\it subconstituent algebra of $\Gamma$ 
 with respect to $x$ and $A$}.
 \end{definition}
 
 \noindent By construction, the algebra $T$ has finite dimension. 

\begin{lemma} \label{lem:TAA} Assume that $\Gamma$ has a dual adjacency matrix $A^*$ with respect to $x$ and $\lbrace E_i \rbrace_{i=0}^{\mathcal D}$.
 Then
the algebra $T$ is generated by $A, A^*$.
 \end{lemma}
 \begin{proof} The algebra $T$ is generated by $M$ and $M^*$. The algebra $M$ is generated by $A$, and the algebra $M^*$ is 
 generated by $A^*$.
 \end{proof}
 
 \noindent Next we give some relations in $T$.
 
 \begin{lemma} \label{lem:Trel} We have $E^*_i  A E^*_j = 0$ if $\vert i - j \vert > 1$ $(0 \leq i,j\leq D)$. Assume that $\Gamma$ has a dual  adjacency matrix 
 $A^*$ with respect to $x$ and $\lbrace E_i \rbrace_{i=0}^{\mathcal D}$. Then
 $E_i A^* E_j = 0$ if $\vert i - j \vert > 1$ $(0 \leq i,j\leq \mathcal D)$.
 \end{lemma}
 \begin{proof} This is a routine consequence of
 \eqref{eq:IntroAact} and \eqref{eq:IntroAsact}.
 \end{proof}

\medskip
\noindent Next we consider the $T$-modules.
By a {\it T-module}
we mean a subspace $W \subseteq V$ such that $BW \subseteq W$
for all $B \in T$. A $T$-module $W$ is called {\it irreducible} whenever $W\not=0$ and $W$ does not contain a $T$-module besides 0 and $W$.
\medskip

\noindent For the rest of this section, we assume that $\Gamma$ has a dual  adjacency matrix 
 $A^*$ with respect to $x$ and $\lbrace E_i \rbrace_{i=0}^{\mathcal D}$.
Let $W$ denote an irreducible $T$-module.
Then $W$ is a direct sum of the nonzero subspaces among $\lbrace E^*_iW\rbrace_{i=0}^D$. Similarly,
$W$ is a direct sum 
 of the nonzero subspaces among $\lbrace E_iW\rbrace_{i=0}^{\mathcal D}$. 
 
 \begin{lemma} \label{lem:WTT} Let $W$ denote an irreducible $T$-module. Then 
\begin{align*}
&A E^*_iW \subseteq E^*_{i-1}W + E^*_iW + E^*_{i+1}W \qquad \qquad (0 \leq i \leq D),
\\
&A^* E_i W \subseteq E_{i-1}W + E_iW+E_{i+1}W \qquad \qquad (0 \leq i \leq {\mathcal D}).
\end{align*}
\end{lemma}
\begin{proof} By \eqref{eq:IntroAact} and \eqref{eq:IntroAsact}.
\end{proof}
 
 \noindent Let $W$ denote an irreducible $T$-module.
By the {\it endpoint} of $W$ we mean
$\mbox{min}\lbrace i \vert 0\leq i \leq D, \; E^*_iW\not=0\rbrace $.
By the {\it diameter} of $W$ we mean
$ |\lbrace i \vert 0 \leq i \leq D,\; E^*_iW\not=0 \rbrace |-1 $.
By the {\it dual endpoint} of $W$ we mean
$\mbox{min}\lbrace i \vert 0\leq i \leq {\mathcal D}, \; E_iW\not=0\rbrace $.
By
the {\it dual diameter} of $W$ we mean
$ |\lbrace i \vert 0 \leq i \leq {\mathcal D},\; E_iW\not=0 \rbrace |-1 $.
%It turns out that the
%diameter of $W$ is  equal to the dual diameter of
%$W$
%\cite[Corollary 3.3]{aap1}.
%By \cite[Lemma 3.4]{terwSub1}
%$\mbox{dim} \,E^*_iW \leq 1$ for $0 \leq i \leq D$ if and only if
%$\mbox{dim} \,E_iW \leq 1$ for $0 \leq i \leq D$; in this case
%$W$ is called {\it thin}. 
\medskip

\noindent The following result is a variation on \cite[Lemma 3.4, Lemma 3.9]{terwSub1}.
\begin{lemma} \label{lem:Wfacts}
Let $W$ denote an irreducible $T$-module with endpoint $r$ and diameter $d$.
Then $r,d$ are nonnegative integers such that $r+d\leq D$.
Moreover the following {\rm (i), (ii)}  hold:
\begin{enumerate}
\item[\rm (i)] 
$E^*_iW \not=0$ if and only if $r \leq i \leq r+d$
$ \quad (0 \leq i \leq D)$;
\item[\rm (ii)]
$W = \sum_{i=r}^{r+d} E^*_{i}W \qquad (\mbox{direct sum}). $
\end{enumerate}
\end{lemma}
\begin{proof} (i) By construction $E^*_r W\not=0$ and $E^*_iW=0$ for $0 \leq i < r$.
Suppose there exists an integer $i$ $(r < i \leq r+d)$ such that $E^*_iW=0$. Define $W'=E^*_r W + E^*_{r+1} W + \cdots + E^*_{i-1}W$.
By construction $0 \not= W' \subseteq  W$. Also by construction, $A^* W'\subseteq W'$. By Lemma \ref{lem:WTT} and $E^*_iW=0$ we obtain
$A W' \subseteq W'$. By these comments $W'$ is a $T$-module. We have $W=W'$ since the $T$-module $W$ is irreducible. This contradicts the
fact that $d$ is the diameter of $W$. We conclude that $E^*_iW\not=0$ for $r \leq i \leq r+d$. By the definition of the diameter $d$ we have
 $E^*_iW=0$ for $r+d < i \leq D$.
 \\
 \noindent (ii) By (i) and the comments above Lemma  \ref{lem:WTT}.
 \end{proof}
\noindent The following is a variation on \cite[Lemma 3.4, Lemma 3.12]{terwSub1}.

\begin{lemma} \label{lem:Wfacts2}
Let $W$ denote an irreducible $T$-module with 
dual endpoint $t$ and dual diameter $\delta$.
Then $t, \delta $ are nonnegative integers such that $t+\delta \leq {\mathcal D}$. Moreover the following {\rm (i), (ii)} hold:
\begin{enumerate}
\item[\rm (i)] 
$E_iW \not=0$ if and only if $t \leq i \leq t+\delta$
$ \quad (0 \leq i \leq {\mathcal D})$;
\item[\rm (ii)]
$W = \sum_{i=t}^{t+\delta} E_{i}W \qquad (\mbox{direct sum}). $
\end{enumerate}
\end{lemma}
\begin{proof} Similar to the proof of Lemma \ref{lem:Wfacts}.
\end{proof}

\noindent  The definition of a tridiagonal pair is given in \cite[Definiton~1.1]{itt}.
The following is a variation on  \cite[Example~1.4]{itt}.

%I %%%\cite[Lemmas~3.9, 3.12]{terwSub1}.

\begin{proposition} \label{prop:TDpair} The pair $A, A^*$ acts on each irreducible $T$-module as a tridiagonal pair.
\end{proposition}
\begin{proof} By  Lemmas \ref{lem:TAA}, \ref{lem:WTT}, \ref{lem:Wfacts}, \ref{lem:Wfacts2}.
\end{proof}

\noindent Let $W$ denote an irreducible $T$-module. 
By Proposition \ref{prop:TDpair} and \cite[Lemma~4.5]{itt},
 the diameter of $W$ is equal to the dual diameter of $W$.
\medskip

\noindent There is a considerable literature about  tridiagonal pairs; see \cite{bbit, itoOq, augIto, IKT, itt, tdanduq, nomTB, ter2, qSerre, aa} and the references therein.
 This literature motivates us to
investigate 
 the $Q$-polynomial property described in Definition \ref{def:Qx}.
  To begin the investigation, we will 
 examine one example in detail.
 This example is constructed from the projective geometry $L_N(q)$.
\medskip

%\noindent See
%\cite{caugh2, caugh5, cerzo, curtin1,curtin2,curtin6,go,hobart,tanabe,terwSub1,terwSub2,terwSub3} for information on the subconstituent
%algebra of a distance-regular graph.

 \section{The projective geometry $L_N(q)$}
 In Section 1 we defined the graph $L_N(q)$. As we will explain in a moment, it is natural to view this graph as the Hasse diagram of a partially ordered set (poset). In order to distinguish between
 the graph and the poset, we will use the following notation going forward.
 \medskip
 
 \noindent
  Given a finite field  ${\rm GF}(q)$ and an integer $N\geq 1$, we define a poset $L_N(q)$ as follows.
  Let $\mathbb V$ denote a vector space over ${\rm GF}(q)$  that has dimension $N$. Let the set $X$ consist of the subspaces of $\mathbb V$. The set $X$, together with the containment relation,
  is a poset denoted $L_N(q)$ and called a projective geometry.
   The partial order is denoted $\leq$. For $y,z \in X$ we write $y < z$ whenever $y \leq z$ and $y \not=z$. 
  We say that {\it $z$ covers $y$} whenever $y < z$ and there does not
  exist $w \in X$ such that $y < w < z$. Note that $z$ covers $y$ if and only if $y \leq z$ and ${\rm dim}\,z - {\rm dim}\,y = 1$.
  Next we define a graph $\Gamma$ with vertex set $X$. Vertices $y,z \in X$ are adjacent in $\Gamma$ whenever one of $y,z $ covers the other one.
  The graph $\Gamma$ is the Hasse diagram of the poset $L_N(q)$. The rest of the paper is about the graph $\Gamma$.
  \medskip
  
  \noindent
  Let $\bf 0$ denote the zero subspace of $\mathbb V$. Recall the vertex $x$ of $\Gamma$ from Section 2. For the rest of the paper, we choose $x={\bf 0}$.
  \medskip 
  
  \noindent Recall the path-length distance function $\partial$ for $\Gamma$.
 \begin{lemma} \label{lem:Lqn} The following {\rm (i)--(iii)} hold for the graph $\Gamma$:
 \begin{enumerate} 
 \item[\rm (i)]  for $y \in X$ we have $\partial({\bf 0}, y)= {\rm dim}\,y$;
 \item[\rm (ii)]  $\Gamma$ has diameter $N$ with respect to the vertex $\bf 0$;
 %\item[\rm (iii)] for $0 \leq i \leq N$,
 %\begin{align*}
 %\Gamma_i({\bf 0}) = \lbrace y \in X| {\rm dim}\, y = i \rbrace;
% \end{align*}
  \item[\rm (iv)] $\Gamma$ is bipartite with bipartition $X=X^+\cup X^-$, where
 \begin{align*}
& X^+ = \lbrace y \in X | {\hbox { ${\rm dim}\,y {\rm \; is \; even}$}} \rbrace, \qquad \quad
 X^- = \lbrace y \in X | {\hbox{ ${\rm dim}\,y {\rm \;is\; odd}$}} \rbrace.
  \end{align*}
  \end{enumerate}
  \end{lemma}
  \noindent For the rest of this paper we adopt the following notation.
  \begin{definition}\rm Define $E^*_i = E^*_i({\bf 0})$ for $0 \leq i \leq N$. Further define $M^*= M^*({\bf 0})$. By construction
  the matrices $\lbrace E^*_i \rbrace_{i=0}^N$ form a basis for $M^*$.
  \end{definition}
  
  \noindent Recall the standard module $V=\mathbb R^X$. By Lemma \ref{lem:Lqn}(i) we obtain
  \begin{align} 
  E^*_iV = {\rm Span}\lbrace \hat y \,\vert\, {\rm dim}\,y = i\rbrace \qquad \qquad (0 \leq i \leq N).
  \label{eq:EsV}
  \end{align}
  
  \noindent For $n \in \mathbb N$ define
  \begin{align*}
  \lbrack n \rbrack_q = \frac{q^n-1}{q-1}.
  \end{align*}
  \noindent We further define
  \begin{align*}
  \lbrack n \rbrack^!_q =  \lbrack n \rbrack_q  \lbrack n-1 \rbrack_q\cdots \lbrack 2 \rbrack_q \lbrack 1 \rbrack_q.
  \end{align*}
  \noindent We interpret $\lbrack 0 \rbrack^!_q =1$. For  $0 \leq i \leq n$ define
  \begin{align*}
  \binom{n}{i}_q =  \frac{\lbrack n \rbrack^!_q}{\lbrack i \rbrack^!_q \lbrack n-i\rbrack^!_q}.
  \end{align*}
  \medskip
  
  \noindent For notational convenience, define $\Gamma_{-1}({\bf 0})=\emptyset$.
  
\begin{lemma} \label{lem:locdr} For $0 \leq i \leq N$ and $y \in \Gamma_i({\bf 0})$ we have
\begin{enumerate}
\item[\rm (i)] $\vert \Gamma(y) \cap \Gamma_{i-1} ({\bf 0}) \vert = \lbrack i \rbrack_q$;
\item[\rm (ii)] $\vert \Gamma(y) \cap \Gamma_{i+1} ({\bf 0}) \vert = \lbrack N-i \rbrack_q$.
\end{enumerate}
\end{lemma}
\begin{proof} By elementary counting arguments; see for example \cite[Section~9.3]{bcn}.
\end{proof}

 \noindent Lemma \ref{lem:locdr} implies that the vertex $\bf 0$ is distance-regularized in the sense of \cite[Section~1.2]{godsil}. \\
 
 \noindent The following result is well known; see for example \cite[Theorem~9.3.2]{bcn}.
 \begin{lemma} \label{cor:ki} For $0 \leq i \leq N$,
 \begin{align*} 
 \vert \Gamma_i({\bf 0}) \vert = \binom{N}{i}_q.
 %%%\left[\begin{array}{c} N  \\
%%%i \end{array} \right]_q.
\end{align*}
\end{lemma}

\noindent % By Lemma \ref{lem:Lqn}(i) and  \eqref{eq:EsV} along with Corollary \ref{cor:ki}, we obtain
By \eqref{DEIV} and Lemma \ref{cor:ki},  
\begin{align}
{\rm dim}\,E^*_iV = \binom{N}{i}_q \qquad \qquad (0 \leq i \leq N).
\label{eq:Eisdim}
\end{align}

\section{The split bases for the standard module $V$}

\noindent We continue to discuss the graph $\Gamma$ from Section 3. Recall that the vectors $\lbrace \hat y\rbrace_{y \in X}$ form a basis for the standard module $V$.
In this section we introduce four additional bases for $V$, said to be split.

\begin{definition} \label{def:4b} \rm For $y \in X $ define
\begin{align}
&y^{\downarrow \downarrow} = \sum_{z \leq y} \hat z,  \label{eq:ydd} 
\\
&y^{\downarrow \uparrow} = \sum_{z \leq y} \hat z (-1)^{{\rm dim} \,z}, \label{eq:ydu} 
\\
&y^{\uparrow \downarrow} = q^{\binom{N-{\rm dim}\,y}{2}} \sum_{ y \leq z} \hat z q^{(N-{\rm dim}\,z){\rm dim}\, y},   \label{eq:yud}  \\
&y^{\uparrow \uparrow} = q^{\binom{N-{\rm dim}\,y}{2}} \sum_{ y \leq z} \hat z q^{(N-{\rm dim}\,z){\rm dim}\, y} (-1)^{{\rm dim}\,z}. \label{eq:yuu} 
%%% \sum_{\stackrel{ \scriptstyle z \in X }{ \scriptstyle z \leq y}} {\hat z},
 \end{align}
 \end{definition}

%\begin{definition}\rm An ordering of $X$ is called {\it standard} whenever for $y, z \in X$, $y$ comes before $z$ if ${\rm dim}\,y < {\rm dim}\,z$.
%\end{definition}

\begin{lemma} \label{lem:4b} Each of following is a basis for the vector space $V$:
\begin{align*}
\lbrace y^{\downarrow \downarrow} \rbrace_{y \in X}, \qquad \quad
\lbrace y^{\downarrow \uparrow} \rbrace_{y \in X}, \qquad \quad
\lbrace y^{\uparrow \downarrow} \rbrace_{y \in X}, \qquad \quad
\lbrace y^{\uparrow \uparrow} \rbrace_{y \in X}.
\end{align*}
\end{lemma}
\begin{proof} The vectors  $\lbrace y^{\downarrow \downarrow} \rbrace_{y \in X}$ are linearly independent by construction, and hence form a basis for $V$.
The remaining assertions are similarly verified.
\end{proof}

\begin{definition}\rm The bases for $V$ from Lemma \ref{lem:4b} are said to be {\it split}.
\end{definition}

\noindent We mention how the split bases are related.
\begin{definition}\label{def:S} \rm Define the diagonal matrix $S \in {\rm Mat}_X(\mathbb R)$ with $(y,y)$-entry 
\begin{align*}
S_{y,y} = (-1)^{{\rm dim}\,y} \qquad \qquad y \in X.
\end{align*}
\end{definition}
\begin{lemma} \label{lem:SS} For $y \in X$ we have $S \hat y = (-1)^{{\rm dim}\,y} \hat y$. 
\end{lemma}
\begin{proof} By  Definition \ref{def:S}. 
\end{proof}
\begin{lemma} \label{lem:S} \rm We have
\begin{align*}
S = \sum_{i=0}^N (-1)^i E^*_i.
\end{align*}
Moreover $S \in M^*$ and $S^2=I$.
\end{lemma}
\begin{proof} The first assertion is immediate from Lemma \ref{lem:Lqn}(i) and Definition \ref{def:S}. The other assertions are clear.
\end{proof}

\begin{lemma} \label{lem:Sw} For $y \in X$ the matrix $S$ sends
\begin{align*} 
 y^{\downarrow \downarrow} \leftrightarrow y^{\downarrow \uparrow}, \qquad \qquad
 y^{\uparrow \downarrow} \leftrightarrow y^{\uparrow \uparrow}.
\end{align*}
\end{lemma}
\begin{proof} By Definition \ref{def:4b} and
Lemma \ref{lem:SS}.
\end{proof}

\section{The matrices $A$ and $A^*$}

We continue to discuss the graph $\Gamma$ from Section 3.
In this section we introduce two matrices $A$ and $A^*$ in ${\rm Mat}_X(\mathbb R)$. The matrix $A$ is a weighted adjacency matrix for $\Gamma$,
and the matrix $A^*$ generates $M^*$.  We investigate the eigenvalues and eigenspaces of $A$ and $A^*$. In the next section, we will explain how $A$ and $A^*$ act on the vectors in
Definition \ref{def:4b}. 
\medskip

\noindent The following matrix $A$ was introduced by S. Ghosh and M. Srinivasan \cite{murali}.
\begin{definition}\label{def:A} \rm (See \cite[Section 1]{murali}.) Define a matrix $A \in {\rm Mat}_X(\mathbb R)$ that has $(y,z)$-entry
\begin{align*}
A_{y,z} = \begin{cases} 1 &{\mbox{\rm if $y$  covers $z$}}; \\
                                      q^{{\rm dim}\,y} &{\mbox{\rm if $z$ covers $y$}}; \\
                                         0 & {\mbox{\rm if $y, z$ are not adjacent}}
                \end{cases} 
                \qquad \qquad y, z \in X.
\end{align*}
\end{definition}
%(A_i)_{y,z} = \begin{cases}  
%1, & {\mbox{\rm if $\partial(y,z)=i$}};\\
%0, & {\mbox{\rm if $\partial(y,z) \ne i$}}
%\end{cases}

\noindent Note that $A$ is a weighted adjacency matrix for $\Gamma$. Next we clarify how $A$ acts on the vectors $\lbrace \hat y \rbrace_{y \in X}$.

\begin{lemma} \label{lem:Ay} For $y \in X$,
\begin{align*}
A \hat y = \sum_{y \;{\rm covers}\; z} \hat z q^{{\rm dim}\, z} + \sum_{z \;{\rm covers}\; y} \hat z.
\end{align*}
\end{lemma}
\begin{proof} By Definition \ref{def:A}.
\end{proof}

\begin{lemma} \label{lem:AsTD} For $0 \leq i \leq N$,
\begin{align*}
A E^*_iV \subseteq E^*_{i-1} V + E^*_{i+1}V.
\end{align*}
%%%%where $E^*_{-1}=0$ and $E^*_{N+1}=0$.
\end{lemma}
\begin{proof} By   \eqref{eq:EsV} and Lemma \ref{lem:Ay}.
\end{proof}

\noindent The following result is due to S.~Ghosh and M.~Srinivasan \cite{murali}. We give a short proof for the sake of completeness.
\begin{lemma} \label{lem:Ad} {\rm (See \cite[Section~1]{murali}.)} The matrix $A$ is diagonalizable.
\end{lemma} 
\begin{proof} The scalar $q $ is positive; let $b$ denote the positive square root of $q$. Define real numbers $\lbrace d_i \rbrace_{i =0}^N$ such that
$d_0=1$ and $d_{i+1}/d_i = b^i$ for $0 \leq i \leq N-1$. Define the matrix 
$\Delta = \sum_{i=0}^N d_i E^*_i$.
The matrix $\Delta$ is invertible, with inverse
$\Delta^{-1} = \sum_{i=0}^N d^{-1}_i E^*_i$.
The matrix
$\Delta A \Delta^{-1} $ has $(y,z)$-entry
\begin{align*}
(\Delta A \Delta^{-1})_{y,z} = \begin{cases} b^{{\rm dim}\,z} &{\mbox{\rm if $y$  covers $z$}}; \\
                                      b^{{\rm dim}\,y} &{\mbox{\rm if $z$ covers $y$}}; \\
                                         0 & {\mbox{\rm if $y, z$ are not adjacent}}
                \end{cases} 
                \qquad \qquad y, z \in X.
\end{align*}
The matrix $\Delta A \Delta^{-1}$ is symmetric. By construction $\Delta A \Delta^{-1}$ has real entries, so $\Delta A \Delta^{-1}$ is
diagonalizable. The  matrices $A$ and $\Delta A \Delta^{-1}$ are similar, so $A$ is diagonalizable.
\end{proof}
\noindent Next we compute the eigenvalues of $A$.

\begin{lemma} \label{lem:Avdd} For $y \in X$,
\begin{align} 
A y^{\downarrow \downarrow} = \frac{q^{{\rm dim}\,y} - q^{N-{\rm dim}\,y}}{q-1} y^{\downarrow \downarrow} + \sum_{z \;{\rm covers}\; y} z^{\downarrow \downarrow}.
\label{eq:Aydd}
\end{align}
\end{lemma}
\begin{proof} This is routinely checked using the combinatorics of $L_N(q)$, see for example \cite[Theorem~9.3.2]{bcn}.
\end{proof}
%\noindent A linear ordering of $X$ is called {\it standard} whenever for all $y, z \in X$,  if ${\rm dim}\,y< {\rm dim}\,z$ then $y$ comes before $z$.

\noindent The following result is due to S.~Ghosh and M.~Srinivasan \cite{murali}. We give a short proof for the sake of completeness.

\begin{lemma}\label{lem:Mur} {\rm (See \cite[Section~2]{murali}.)} The eigenvalues of $A$ are $\lbrace \theta_i \rbrace_{i=0}^N$, where
\begin{align}
\theta_i = \frac{q^{N-i} - q^i}{q-1} \qquad \qquad (0 \leq i \leq N).
\label{eq:thi}
\end{align} Moreover, for $0 \leq i \leq N$ the dimension of the $\theta_i$-eigenspace of the $A$ is equal to $\binom{N}{i}_q$.
\end{lemma}
\begin{proof} 
Linearly order the elements of $X$, such that for all $y, z \in X$ the element $y$ comes before $z$ if ${\rm dim}\,y< {\rm dim}\,z$.
Consider the matrix in ${\rm Mat}_X(\mathbb R)$ that represents $A$ with respect to the basis $\lbrace y^{\downarrow \downarrow} \rbrace_{y \in X}$.
This matrix is lower triangular by Lemma \ref{lem:Avdd}. Also by Lemma \ref{lem:Avdd}, this matrix has diagonal entries $\lbrace \theta_i \rbrace_{i=0}^N$, with $\theta_i$ appearing
 $\binom{N}{i}_q$
 times for $0 \leq i \leq N$. The result follows.
\end{proof} 
\noindent We remark that
\begin{align}
\theta_{N-i} = - \theta_i \qquad \qquad (0 \leq i \leq N).
\label{eq:th}
\end{align}

\noindent Recall from Section 2 the adjacency algebra $M$ generated by $A$. 
By Lemmas \ref{lem:Ad}, \ref{lem:Mur} 
the vector space $M$ has a basis $\lbrace E_i \rbrace_{i=0}^N$ of primitive idempotents, labelled such that $A= \sum_{i=0}^N \theta_i E_i$. The subspace
$E_iV$ is the $\theta_i$-eigenspace of $A$ for $0 \leq i \leq N$. By this and Lemma \ref{lem:Mur},
\begin{align}
{\rm dim}\, E_iV = \binom{N}{i}_q \qquad \qquad (0 \leq i \leq N).
\label{eq:Eidim}
\end{align}
\noindent Recall the integer  $\mathcal D$ from below Definition \ref{def:wa}.
We have $\mathcal D=N$.

\begin{definition}\label{def:As} \rm Define a diagonal matrix $A^* \in {\rm Mat}_X(\mathbb R)$ with $(y,y)$-entry
\begin{align*}
A^*_{y,y} = q^{-{\rm dim}\,y} \qquad \qquad y \in X.
\end{align*}
\end{definition}
\begin{lemma} \label{lem:AsEs} We have
\begin{align*}
A^* = \sum_{i=0}^N q^{-i} E^*_i.
\end{align*}
Moreover, $A^*$ generates $M^*$.
\end{lemma} 
\begin{proof} The first assertion follows from Lemma \ref{lem:Lqn}(i) and Definition \ref{def:As}. The second assertion follows from the first assertion and the fact that $\lbrace q^{-i}\rbrace_{i=0}^N$
are mutually distinct.
\end{proof}

\begin{lemma} The eigenvalues of $A^*$ are $\lbrace \theta^*_i\rbrace_{i=0}^N$, where
\begin{align}
\theta^*_i= q^{-i} \qquad \qquad (0 \leq i \leq N). \label{eq:thsi}
\end{align}
Moreover, for $0 \leq i \leq N$ the $\theta^*_i$-eigenspace of $A^*$ is equal to $E^*_iV$. 
\end{lemma}
\begin{proof} By Lemma \ref{lem:AsEs}.
\end{proof}

\noindent Recall the matrix $S$ from Definition \ref{def:S}.

\begin{lemma} \label{lem:SAS} We have
\begin{align*}
S A S^{-1} = - A , \qquad \qquad SA^*S^{-1}= A^* .
\end{align*}
\end{lemma}
\begin{proof} The first equation is verified  using
Lemmas  \ref{lem:S}, \ref{lem:AsTD}. The second equation holds since $S$ and $A^*$ are diagonal.
\end{proof} 

\begin{lemma} For $0 \leq i \leq N$ we have
\begin{align}
S E_i S^{-1} = E_{N-i}, \qquad \qquad S E^*_i S^{-1} = E^*_i.\label{eq:SES}
\end{align}
Moreover
\begin{align}
 S E_iV = E_{N-i}V, \qquad \qquad  S E^*_iV = E^*_iV. \label{eq:SEV}
 \end{align}
 \end{lemma} 
 \begin{proof} By \eqref{eq:th} and Lemma \ref{lem:SAS}.
 \end{proof}

\section{The action of $A$ and $A^*$ on the split bases}

\noindent We continue to discuss the graph $\Gamma$ from Section 3. Recall the four split bases of $V$, from Lemma  \ref{lem:4b}. In this 
section we compute the action of $A$ and $A^*$ on these bases.

\begin{lemma} \label{lem:actiondd} For $0 \leq i \leq N$ and $y \in \Gamma_i({\bf 0})$ we have
\begin{align*}
A y^{\downarrow \downarrow} &= \theta_{N-i} y^{\downarrow \downarrow} + \sum_{z \;{\rm covers}\; y} z^{\downarrow \downarrow},
\\
A^* y^{\downarrow \downarrow} &= \theta^*_i y^{\downarrow \downarrow} + (q-1)q^{-i} \sum_{y \;{\rm covers}\; z} z^{\downarrow \downarrow}.
\end{align*}
\end{lemma}
\begin{proof} The first equation is a reformulation of
\eqref{eq:Aydd}. The second equation is routinely obtained using \eqref{eq:ydd} and Definition \ref{def:As}. 
\end{proof}

\begin{lemma} \label{lem:actiondu} For $0 \leq i \leq N$ and $y \in \Gamma_i({\bf 0})$ we have
\begin{align*}
A y^{\downarrow \uparrow} &= \theta_{i} y^{\downarrow \uparrow} - \sum_{z \;{\rm covers}\; y} z^{\downarrow \uparrow},
\\
A^* y^{\downarrow \uparrow} &= \theta^*_i y^{\downarrow \uparrow} + (q-1)q^{-i} \sum_{y \;{\rm covers}\; z} z^{\downarrow \uparrow}.
\end{align*}
\end{lemma}
\begin{proof} For the equations in Lemma \ref{lem:actiondd}, apply $S$ to each side and evaluate the result using \eqref{eq:th} along with Lemmas \ref{lem:Sw}, \ref{lem:SAS}.
\end{proof}

\begin{lemma} \label{lem:actionud} For $0 \leq i \leq N$ and $y \in \Gamma_{i}({\bf 0})$ we have
\begin{align*}
A y^{\uparrow \downarrow} &= \theta_{i} y^{\uparrow \downarrow} + \sum_{y \;{\rm covers}\; z} z^{\uparrow \downarrow},
\\
A^* y^{\uparrow \downarrow} &= \theta^*_{i} y^{\uparrow \downarrow} + (q^{-1}-1)q^{-i} \sum_{z \;{\rm covers}\; y} z^{\uparrow \downarrow}.
\end{align*}
\end{lemma}
\begin{proof} The equations are routinely verified using the combinatorics of $L_N(q)$, see for example \cite[Theorem~9.3.2]{bcn}.
\end{proof}

\begin{lemma} \label{lem:actionuu} For $0 \leq i \leq N$ and $y \in \Gamma_{i}({\bf 0})$ we have
\begin{align*}
A y^{\uparrow \uparrow} &= \theta_{N-i} y^{\uparrow \uparrow} - \sum_{y \;{\rm covers}\; z} z^{\uparrow \uparrow},
\\
A^* y^{\uparrow \uparrow} &= \theta^*_{i} y^{\uparrow \uparrow} + (q^{-1}-1)q^{-i} \sum_{z \;{\rm covers}\; y} z^{\uparrow \uparrow}.
\end{align*}
\end{lemma}
\begin{proof} For the equations in Lemma \ref{lem:actionud}, apply $S$ to each side and evaluate the result using \eqref{eq:th} along with Lemmas \ref{lem:Sw}, \ref{lem:SAS}.
\end{proof}

\section{The split decompositions of the standard module $V$}

\noindent We continue to discuss the graph $\Gamma$ from Section 3. Recall the standard module $V$. By a {\it decomposition of $V$} we mean 
a sequence of nonzero subspaces $\lbrace U_i \rbrace_{i=0}^N$ whose direct sum is equal to $V$. For example,
the sequences $\lbrace E_iV\rbrace_{i=0}^N$ and $\lbrace E^*_iV\rbrace_{i=0}^N$ are decompositions of $V$. In this section
we introduce four additional decompositions of $V$, said to be split. We discuss how the split decompositions are related to
 $\lbrace E_iV\rbrace_{i=0}^N$ and $\lbrace E^*_iV\rbrace_{i=0}^N$.
We also discuss how $A$ and $A^*$ act on the split decompositions.

\begin{lemma} \label{lem:V1} For following hold for $0 \leq i \leq N$.
\begin{enumerate}
\item[\rm (i)] The vectors $\lbrace y^{\downarrow \downarrow} \vert y \in X, \; {\rm dim}\,y \leq i\rbrace$ form a basis for $E^*_0V+E^*_1V+\cdots + E^*_iV$.
\item[\rm (ii)] The vectors $\lbrace y^{\downarrow \downarrow} \vert y \in X, \; {\rm dim}\,y \geq i\rbrace$ form a basis for $E_0V+E_1V+\cdots + E_{N-i}V$.
\end{enumerate}
\end{lemma}
\begin{proof} (i) Use   \eqref{eq:EsV} along with Definition \ref{def:4b} and Lemma \ref{lem:4b}.
\\
\noindent (ii) We claim  that the given vectors are contained in the given subspace. To prove the claim, pick $y \in X$ such that ${\rm dim}\,y \geq i$. 
For notational convenience, define $j={\rm dim}\,y$.
Using Lemma \ref{lem:actiondd} repeatedly, we obtain
\begin{align*}
(A-\theta_{0}I)(A-\theta_{1}I) \cdots (A-\theta_{N-j}I) y^{\downarrow \downarrow} = 0.
\end{align*}
Consequently 
\begin{align*}
y^{\downarrow\downarrow} \in E_0V+E_1V+\cdots + E_{N-j}V \subseteq E_0V+E_1V+\cdots + E_{N-i}V.
\end{align*}
The claim is proved. By Lemma \ref{lem:4b} the given vectors are linearly independent. By Lemma  \ref{cor:ki} 
and  \eqref{eq:Eidim}, the
 number of given vectors is equal to the dimension of the
given subspace. The result follows from these comments and linear algebra.
\end{proof}

\begin{lemma} \label{lem:V2} For following hold for $0 \leq i \leq N$.
\begin{enumerate}
\item[\rm (i)] The vectors $\lbrace y^{\downarrow \uparrow} \vert y \in X, \; {\rm dim}\,y \leq i\rbrace$ form a basis for $E^*_0V+E^*_1V+\cdots + E^*_iV$.
\item[\rm (ii)] The vectors $\lbrace y^{\downarrow \uparrow} \vert y \in X, \; {\rm dim}\,y \geq i\rbrace$ form a basis for $E_NV+E_{N-1}V+\cdots + E_iV$.
\end{enumerate}
\end{lemma}
\begin{proof} Apply $S$ to everything in Lemma \ref{lem:V1}, and evaluate the result using
Lemma \ref{lem:Sw} 
and
\eqref{eq:SEV}.
\end{proof}

\begin{lemma} \label{lem:V3} For following hold for $0 \leq i \leq N$.
\begin{enumerate}
\item[\rm (i)] The vectors $\lbrace y^{\uparrow \downarrow} \vert y \in X, \; {\rm dim}\,y \geq i\rbrace$ form a basis for $E^*_NV+E^*_{N-1}V+\cdots + E^*_iV$.
\item[\rm (ii)] The vectors $\lbrace y^{\uparrow \downarrow} \vert y \in X, \; {\rm dim}\,y \leq i\rbrace$ form a basis for $E_0V+E_1V+\cdots + E_{i}V$.
\end{enumerate}
\end{lemma}
\begin{proof} (i) Use   \eqref{eq:EsV} along with Definition \ref{def:4b} and Lemma \ref{lem:4b}.
\\
\noindent (ii) We claim  that the given vectors are contained in the given subspace. To prove the claim, pick $y \in X$ such that ${\rm dim}\,y \leq i$. 
For notational convenience, define $j={\rm dim}\,y$.
Using Lemma \ref{lem:actionud} repeatedly, we obtain
\begin{align*}
(A-\theta_{0}I)(A-\theta_{1}I) \cdots (A-\theta_{j}I) y^{\uparrow \downarrow} = 0.
\end{align*}
Consequently 
\begin{align*}
y^{\uparrow\downarrow} \in E_0V+E_1V+\cdots + E_{j}V \subseteq E_0V+E_1V+\cdots + E_{i}V.
\end{align*}
The claim is proved. By Lemma \ref{lem:4b} the given vectors are linearly independent. By Lemma  \ref{cor:ki} 
and  \eqref{eq:Eidim}, the
 number of given vectors is equal to the dimension of the
given subspace. The result follows from these comments and linear algebra.
\end{proof}

\begin{lemma} \label{lem:V4} For following hold for $0 \leq i \leq N$.
\begin{enumerate}
\item[\rm (i)] The vectors $\lbrace y^{\uparrow \uparrow} \vert y \in X, \; {\rm dim}\,y \geq i\rbrace$ form a basis for $E^*_NV+E^*_{N-1}V+\cdots + E^*_iV$.
\item[\rm (ii)] The vectors $\lbrace y^{\uparrow \uparrow} \vert y \in X, \; {\rm dim}\,y \leq i\rbrace$ form a basis for $E_NV+E_{N-1}V+\cdots + E_{N-i}V$.
\end{enumerate}
\end{lemma}
\begin{proof} Apply $S$ to everything in Lemma \ref{lem:V3}, and evaluate the result using
Lemma \ref{lem:Sw} 
and
\eqref{eq:SEV}.
\end{proof}

\begin{definition}\label{def:u} For $0 \leq i \leq N$ we define
\begin{align*}
U^{\downarrow \downarrow}_i &= (E^*_0V+ E^*_1V+ \cdots + E^*_iV)\cap (E_0V + E_1V+ \cdots + E_{N-i}V), \\
U^{\downarrow \uparrow}_i &= (E^*_0V+ E^*_1V+ \cdots + E^*_iV)\cap (E_NV + E_{N-1}V+ \cdots + E_{i}V), \\
U^{\uparrow \downarrow}_i &= (E^*_NV+ E^*_{N-1}V+ \cdots + E^*_{N-i}V)\cap (E_0V + E_1V+ \cdots + E_{N-i}V), \\
U^{\uparrow \uparrow}_i &= (E^*_NV+ E^*_{N-1}V+ \cdots + E^*_{N-i}V)\cap (E_NV + E_{N-1}V+ \cdots + E_{i}V).
\end{align*}
\end{definition}

\begin{lemma} \label{lem:table} The following {\rm (i)--(iv)} hold 
for $0 \leq i \leq N$:
\begin{enumerate}
\item[\rm (i)] the vectors $\lbrace y^{\downarrow \downarrow}\rbrace_{y \in \Gamma_i({\bf 0})}$ form a basis for  $U^{\downarrow \downarrow}_i$;
\item[\rm (ii)]  the vectors $\lbrace y^{\downarrow \uparrow}\rbrace_{y \in \Gamma_i({\bf 0})}$ form a basis for  $U^{\downarrow \uparrow}_i$;
\item[\rm (iii)] the vectors $\lbrace y^{\uparrow \downarrow}\rbrace_{y \in \Gamma_{N-i}({\bf 0})}$ form a basis for  $U^{\uparrow \downarrow}_i$;
\item[\rm (iv)] the vectors $\lbrace y^{\uparrow \uparrow}\rbrace_{y \in \Gamma_{N-i}({\bf 0})}$ form a basis for $U^{\uparrow \uparrow}_i$.
\end{enumerate}
\end{lemma}
\begin{proof} Recall Lemma \ref{lem:Lqn}(i) and Lemma \ref{lem:4b}. \\
\noindent (i) Compare  the two assertions in Lemma \ref{lem:V1}.\\
\noindent (ii) Compare  the two assertions in Lemma \ref{lem:V2}.\\
\noindent (iii) Compare  the two assertions in Lemma \ref{lem:V3}.\\
\noindent (iv) Compare  the two assertions in Lemma \ref{lem:V4}.
\end{proof}

\begin{lemma} For $0 \leq i \leq N$ the following subspaces have dimension $\binom{N}{i}_q$:
\begin{align*}
 U^{\downarrow \downarrow}_i, \qquad  U^{\downarrow \uparrow}_i, \qquad   U^{\uparrow \downarrow}_i, \qquad  U^{\uparrow \uparrow}_i.
\end{align*}
\end{lemma}
\begin{proof} By Lemma
 \ref{cor:ki} and Lemma \ref{lem:table} along with $\binom{N}{i}_q = \binom{N}{N-i}_q$.
 \end{proof}

\begin{lemma}\label{lem:4dec} Each of the following is a decomposition of $V$:
\begin{align*}
\lbrace U^{\downarrow \downarrow}_i \rbrace_{i=0}^N, \qquad \quad
\lbrace U^{\downarrow \uparrow}_i \rbrace_{i=0}^N, \qquad  \quad
\lbrace U^{\uparrow \downarrow}_i \rbrace_{i=0}^N, \qquad \quad
\lbrace U^{\uparrow \uparrow}_i \rbrace_{i=0}^N.
\end{align*}
\end{lemma}
\begin{proof} By Lemmas  \ref{lem:4b}, \ref{lem:table}.
\end{proof}

\noindent The following definition is motivated by \cite[Section 4]{itt} and \cite[Section~5]{splitDec}; see also \cite{qtet, kim1, kim2}.

\begin{definition}\rm The decompositions of $V$ from Lemma \ref{lem:4dec} are said to be {\it split}.
\end{definition}

\begin{lemma} \label{lem:Ss} For $0 \leq i \leq N$ we have
\begin{align*}
S U^{\downarrow \downarrow}_i = U^{\downarrow \uparrow}_i,  \qquad \quad
S U^{\downarrow \uparrow}_i = U^{\downarrow \downarrow}_i, \qquad \quad 
S U^{\uparrow \downarrow}_i = U^{\uparrow \uparrow}_i, \qquad \quad
S U^{\uparrow \uparrow}_i = U^{\uparrow \downarrow}_i.
\end{align*}
\end{lemma} 
\begin{proof} By Lemmas  \ref{lem:Sw}, \ref{lem:table}.
\end{proof}

\begin{lemma} \label{lem:3one} For $0 \leq i \leq N$ the following are equal:
\begin{align*}
E^*_0V+ E^*_1V+\cdots + E^*_iV, \qquad  U^{\downarrow \downarrow}_0 + U^{\downarrow \downarrow}_1 + \cdots + U^{\downarrow \downarrow}_i, \qquad
 U^{\downarrow \uparrow}_0 + U^{\downarrow \uparrow}_1 + \cdots + U^{\downarrow \uparrow}_i.
\end{align*}
\end{lemma}
\begin{proof} By Lemma \ref{lem:V1}(i)
the vectors $\lbrace y^{\downarrow \downarrow} \vert {\dim}\,y \leq i\rbrace$ form a basis for $E^*_0V+\cdots + E^*_iV$. By Lemma \ref{lem:table}(i)
the vectors $\lbrace y^{\downarrow \downarrow} \vert {\dim}\,y \leq i\rbrace$ form a basis for 
$U^{\downarrow \downarrow}_0 + \cdots + U^{\downarrow \downarrow}_i$. By these comments
$E^*_0V+\cdots + E^*_iV=U^{\downarrow \downarrow}_0  + \cdots + U^{\downarrow \downarrow}_i$.
 By Lemma \ref{lem:V2}(i)
the vectors $\lbrace y^{\downarrow \uparrow} \vert {\dim}\,y \leq i\rbrace$ form a basis for $E^*_0V+\cdots + E^*_iV$. By Lemma \ref{lem:table}(ii)
the vectors $\lbrace y^{\downarrow \uparrow} \vert {\dim}\,y \leq i\rbrace$ form a basis for 
$U^{\downarrow \uparrow}_0 + \cdots + U^{\downarrow \uparrow}_i$. By these comments
$E^*_0V+\cdots + E^*_iV=U^{\downarrow \uparrow}_0  + \cdots + U^{\downarrow \uparrow}_i$.
\end{proof}

\begin{lemma} \label{lem:3two}  For $0 \leq i \leq N$ the following are equal:
\begin{align*}
E^*_NV+ E^*_{N-1}V+\cdots + E^*_{N-i}V, \qquad  U^{\uparrow \downarrow}_0 + U^{\uparrow \downarrow}_1 + \cdots + U^{\uparrow \downarrow}_i, \qquad
 U^{\uparrow \uparrow}_0 + U^{\uparrow \uparrow}_1 + \cdots + U^{\uparrow \uparrow}_i.
\end{align*}
\end{lemma}
\begin{proof} The proof is similar to the proof of Lemma \ref{lem:3one}.
 By Lemma \ref{lem:V3}(i)
the vectors $\lbrace y^{\uparrow \downarrow} \vert {\dim}\,y \geq N-i\rbrace$ form a basis for $E^*_NV+\cdots + E^*_{N-i}V$. By Lemma \ref{lem:table}(iii)
the vectors $\lbrace y^{\uparrow \downarrow} \vert {\dim}\,y \geq N-i\rbrace$ form a basis for 
$U^{\uparrow \downarrow}_0 + \cdots + U^{\uparrow \downarrow}_i$. By these comments
$E^*_NV+\cdots + E^*_{N-i}V=U^{\uparrow \downarrow}_0  + \cdots + U^{\uparrow \downarrow}_i$.
 By Lemma \ref{lem:V4}(i)
the vectors $\lbrace y^{\uparrow \uparrow} \vert {\dim}\,y \geq N-i\rbrace$ form a basis for $E^*_NV+\cdots + E^*_{N-i}V$. By Lemma \ref{lem:table}(iv)
the vectors $\lbrace y^{\uparrow \uparrow} \vert {\dim}\,y \geq N-i\rbrace$ form a basis for 
$U^{\uparrow \uparrow}_0 + \cdots + U^{\uparrow \uparrow}_i$. By these comments
$E^*_NV+\cdots + E^*_{N-i}V=U^{\uparrow \uparrow}_0  + \cdots + U^{\uparrow \uparrow}_i$.
\end{proof}

\begin{lemma} \label{lem:3three} For $0 \leq i \leq N$ the following are equal:
\begin{align*}
E_0V+ E_1V+\cdots + E_iV, \qquad  U^{\downarrow \downarrow}_N + U^{\downarrow \downarrow}_{N-1} + \cdots + U^{\downarrow \downarrow}_{N-i}, \qquad
 U^{\uparrow \downarrow}_N + U^{\uparrow \downarrow}_{N-1} + \cdots + U^{\uparrow \downarrow}_{N-i}.
\end{align*}
\end{lemma}
\begin{proof} The proof is similar to the proof of Lemma \ref{lem:3one}.
 By Lemma \ref{lem:V1}(ii)
the vectors $\lbrace y^{\downarrow \downarrow} \vert {\dim}\,y \geq N-i\rbrace$ form a basis for $E_0V+\cdots + E_{i}V$. By Lemma \ref{lem:table}(i)
the vectors $\lbrace y^{\downarrow \downarrow} \vert {\dim}\,y \geq N-i\rbrace$ form a basis for 
$U^{\downarrow \downarrow}_N + \cdots + U^{\downarrow \downarrow}_{N-i}$. By these comments
$E_0V+\cdots + E_iV=U^{\downarrow \downarrow}_N  + \cdots + U^{\downarrow \downarrow}_{N-i}$.
 By Lemma \ref{lem:V3}(ii)
the vectors $\lbrace y^{\uparrow \downarrow} \vert {\dim}\,y \leq i\rbrace$ form a basis for $E_0V+\cdots + E_iV$. By Lemma \ref{lem:table}(iii)
the vectors $\lbrace y^{\uparrow \downarrow} \vert {\dim}\,y \leq i\rbrace$ form a basis for 
$U^{\uparrow \downarrow}_N + \cdots + U^{\uparrow \downarrow}_{N-i}$. By these comments
$E_0V+\cdots + E_iV=U^{\uparrow \downarrow}_N  + \cdots + U^{\uparrow \downarrow}_{N-i}$.
\end{proof}

\begin{lemma} \label{lem:3four} For $0 \leq i \leq N$ the following are equal:
\begin{align*}
E_NV+ E_{N-1}V+\cdots + E_{N-i}V, \qquad  U^{\downarrow \uparrow}_N + U^{\downarrow \uparrow}_{N-1} + \cdots + U^{\downarrow \uparrow}_{N-i}, \qquad
 U^{\uparrow \uparrow}_N + U^{\uparrow \uparrow}_{N-1} + \cdots + U^{\uparrow \uparrow}_{N-i}.
\end{align*}
\end{lemma}
\begin{proof} Apply $S$ to everything in Lemma \ref{lem:3three}, and evaluate the result using \eqref{eq:SEV}
and Lemma  \ref{lem:Ss}.
\end{proof}

\noindent We make a definition for notational convenience. For a decomposition $\lbrace U_i \rbrace_{i=0}^N$ of $V$, define $U_{-1}=0$ and $U_{N+1}=0$.

\begin{lemma} \label{lem:AU} For $0 \leq i \leq N$ we have
\begin{align} \label{eq:in1}
(A-\theta_{N-i}I) U^{\downarrow \downarrow}_i I &\subseteq U^{\downarrow \downarrow}_{i+1}, \qquad \qquad \quad
(A^*-\theta^*_i I) U^{\downarrow \downarrow}_i \subseteq U^{\downarrow \downarrow}_{i-1},\\
\label{eq:in2}
(A-\theta_{i}I) U^{\downarrow \uparrow}_i &\subseteq U^{\downarrow \uparrow}_{i+1}, \qquad \qquad \quad
(A^*-\theta^*_i I) U^{\downarrow \uparrow}_i \subseteq U^{\downarrow \uparrow}_{i-1},\\
\label{eq:in3}
(A-\theta_{N-i}I) U^{\uparrow \downarrow}_i &\subseteq U^{\uparrow \downarrow}_{i+1}, \qquad \qquad
(A^*-\theta^*_{N-i} I) U^{\uparrow \downarrow}_i \subseteq U^{\uparrow \downarrow}_{i-1},\\
\label{eq:in4}
(A-\theta_{i}I) U^{\uparrow \uparrow}_i &\subseteq U^{\uparrow \uparrow}_{i+1}, \qquad \qquad
(A^*-\theta^*_{N-i} I) U^{\uparrow \uparrow}_i \subseteq U^{\uparrow \uparrow}_{i-1}.
\end{align}
\end{lemma}
\begin{proof} The inclusions  \eqref{eq:in1} follow from Lemma  \ref{lem:actiondd} and Lemma \ref{lem:table}(i). 
The inclusions  \eqref{eq:in2} follow from Lemma  \ref{lem:actiondu} and Lemma \ref{lem:table}(ii). 
The inclusions  \eqref{eq:in3} follow from Lemma  \ref{lem:actionud} and Lemma \ref{lem:table}(iii). 
The inclusions  \eqref{eq:in4} follow from Lemma  \ref{lem:actionuu} and Lemma \ref{lem:table}(iv). 
\end{proof}

\noindent We finish this section with a comment. Let $0 \leq i \leq N$. In  \eqref{eq:EsV}
we gave a basis for $E^*_iV$. In Lemma \ref{lem:table} we gave a basis for  each of $U^{\downarrow \downarrow}_i$, $U^{\downarrow \uparrow}_i$,
 $U^{\uparrow \downarrow}_i$,
 $U^{\uparrow \uparrow}_i$. 
A basis for $E_iV$ can be found in \cite[Section~4]{murali}.

\section{The matrix $A$ is $Q$-polynomial with respect to $\bf 0$}
 We continue to discuss the graph $\Gamma$ from Section 3. Recall the weighted adjacency matrix $A$ from Definition \ref{def:A},
 and the diagonal matrix $A^*$ from Definition \ref{def:As}.
 In this section we show that $A^*$ is a dual adjacency matrix of $\Gamma$ with respect to $\bf 0$ and the ordering $\lbrace E_i\rbrace_{i=0}^N$.
 Using this fact, we show that 
  $A$ is $Q$-polynomial with respect to $\bf 0$. We consider the subconstitutent algebra $T=T({\bf 0}, A)$. We show that the generators $A, A^*$ satisfy 
 a pair of relations called the tridiagonal relations. We describe the irreducible $T$-modules.
 \medskip

 \noindent Recall the standard module $V$.

\begin{proposition}\label{lem:AT} For $0 \leq i \leq N$,
\begin{align}
\label{eq:3T}
A^* E_iV \subseteq E_{i-1}V+ E_iV+E_{i+1}V.
\end{align}
%%where $E_{-1}=0$ and $E_{N+1}=0$.
\end{proposition}
\begin{proof} Using Lemma  \ref{lem:3three} and \eqref{eq:in1}, we obtain
\begin{align*}
A^* E_iV &\subseteq A^* (E_0V+ \cdots + E_iV) \\
&= A^* (U^{\downarrow \downarrow}_N + \cdots + U^{\downarrow \downarrow}_{N-i}) \\
&\subseteq 
U^{\downarrow \downarrow}_N + \cdots + U^{\downarrow \downarrow}_{N-i-1} \\
&= E_0V+\cdots + E_{i+1}V.
\end{align*}
\noindent Using Lemma \ref{lem:3four} and \eqref{eq:in4}, we obtain
\begin{align*}
A^* E_iV &\subseteq A^* (E_iV+ \cdots + E_NV) \\
&= A^* (U^{\uparrow \uparrow}_i + \cdots + U^{\uparrow \uparrow}_{N}) \\
&\subseteq 
U^{\uparrow \uparrow}_{i-1} + \cdots + U^{\uparrow \uparrow}_{N} \\
&= E_{i-1}V+\cdots + E_{N}V.
\end{align*}
By the above comments
\begin{align*}
A^* E_iV &\subseteq (E_0V+\cdots + E_{i+1}V)\cap (E_{i-1}V+\cdots + E_NV) \\
&= E_{i-1}V+ E_iV+E_{i+1}V.
\end{align*}
\end{proof}

\begin{corollary} \label{cor:dam} The matrix $A^*$ is a dual adjacency matrix of $\Gamma$ with respect to the vertex $\bf 0$ and the ordering $\lbrace E_i\rbrace_{i=0}^N$.
\end{corollary}
\begin{proof} By Definition \ref{def:dualA}, Lemma \ref{lem:AsEs}, and Proposition \ref{lem:AT}.
\end{proof}

\begin{corollary} \label{cor:Qpoly} The  ordering $\lbrace E_i\rbrace_{i=0}^N$ is $Q$-polynomial with respect to the vertex $\bf 0$.
\end{corollary}
\begin{proof} By Definition \ref{def:Qpoly} and Corollary \ref{cor:dam}.
\end{proof}
\noindent The following is the main result of the paper.

\begin{theorem} \label{thm:main1}The matrix $A$ is $Q$-polynomial with respect to the vertex $\bf 0$.
\end{theorem}
\begin{proof} By Definition \ref{def:Qx} and Corollary \ref{cor:Qpoly}.
\end{proof}

\noindent For the rest of the paper we adopt the following notation.
\begin{definition}\label{def:T} \rm Let $T=T({\bf 0}, A)$ denote %%the subalgebra of ${\rm Mat}_X(\mathbb R)$ generated by $M$ and $M^*$. We call $T$
the subconstituent algebra of $\Gamma$ with respect to $\bf 0$ and $A$.
\end{definition}
\noindent By Lemma \ref{lem:TAA}, the algebra $T$ is generated by $A, A^*$. Next we will display two relations satisfied by these generators.

\begin{proposition} \label{thm:TD} The matrices $A$ and $A^*$ satisfy
\begin{align*}
&A^3 A^* - (q+q^{-1}+1) A^2 A^* A+(q+q^{-1}+1)AA^*A^2 - A^*A^3 \\
&\qquad \quad= q^{N-2}(q+1)^2 (AA^*-A^*A),\\
&A^{*3} A-(q+q^{-1}+1) A^{*2}AA^* + (q+q^{-1}+1) A^*A A^{*2} - A A^{*3} = 0.
\end{align*}
\end{proposition}
\begin{proof} Concerning the first equation, let $C_1$ denote the left-hand side minus the right-hand side. We show that $C_1=0$. We have
\begin{align*}
 C_1 = I C_1 I = \sum_{i=0}^N \sum_{j=0}^N E_i C_1 E_j.
 \end{align*}
 \noindent For $0 \leq i,j\leq N$ we show that $E_i C_1 E_j = 0$. Using $E_i A = \theta_i E_i$ and $A E_j = \theta_j E_j$, we obtain
 \begin{align*}
 E_iC_1E_j & = E_i A^* E_j \Bigl( \theta^3_i - (q+q^{-1} + 1) \theta^2_i \theta_j + (q+ q^{-1}+1) \theta_i \theta^2_j - \theta^3_j - q^{N-2} (q+1)^2 (\theta_i - \theta_j)\Bigr)
 \\
 &= E_i A^* E_j (\theta_i - \theta_j)\bigl( \theta^2_i - (q+q^{-1}) \theta_i \theta_j + \theta^2_j - q^{N-2} (q+1)^2\bigr).
 \end{align*}
 \noindent We examine the factors in the above line. 
 By Lemma \ref{lem:Trel} we have $E_iA^* E_j=0$ if $\vert i-j\vert > 1$. Of course $\theta_i - \theta_j = 0$ if $i=j$. Using \eqref{eq:thi} we obtain
 \begin{align*}
 \theta^2_i - (q+q^{-1}) \theta_i \theta_j + \theta^2_j - q^{N-2} (q+1)^2=0 \qquad \qquad {\mbox{\rm if $\vert i - j \vert = 1$.}}
 \end{align*}
 \noindent By these comments $E_iC_1E_j=0$.
 We have shown that $C_1=0$, so  the first equation holds. Concerning the second equation, let $C_2$ denote the left-hand side. We have
 \begin{align*}
 C_2 = I C_2 I = \sum_{i=0}^N \sum_{j=0}^N E^*_i C_2 E^*_j.
 \end{align*}
 \noindent For $0 \leq i,j\leq N$ we show that $E^*_i C_2 E^*_j = 0$. Using $E^*_i A^* = \theta^*_i E^*_i$ and $A^* E^*_j = \theta^*_j E^*_j$, we obtain
 \begin{align*}
 E^*_iC_2E^*_j & = E^*_i A E^*_j \Bigl( \theta^{*3}_i - (q+q^{-1} + 1) \theta^{*2}_i \theta^*_j + (q+ q^{-1}+1) \theta^*_i \theta^{*2}_j - \theta^{*3}_j \Bigr)
 \\
 &= E^*_i A E^*_j (\theta^*_i - \theta^*_j)(\theta^*_i- q \theta^*_j)(\theta^*_i - q^{-1} \theta^*_j).
 \end{align*}
 \noindent We examine the factors in the above line.
 By Lemma \ref{lem:Trel} we have $E^*_iAE^*_j=0$ if $\vert i-j\vert > 1$. Of course $\theta^*_i - \theta^*_j = 0$ if $i=j$. By \eqref{eq:thsi} we 
 have $\theta^*_i - q \theta^*_j=0$ if $j-i=1$ and $\theta^*_i - q^{-1} \theta^*_j = 0 $ if $i-j=1$.
 \noindent By these comments $E^*_iC_2E^*_j=0$.
 We have shown that $C_2=0$, so  the second equation holds. 
 \end{proof}

\begin{note}\rm The equations in Proposition \ref{thm:TD} are an instance of the tridiagonal relations \cite{qSerre}.
\end{note}

\noindent  Next we discuss the representation theory of $T$. Before we go into detail, we have some remarks about $L_N(q)$.
The poset $L_N(q)$ is uniform in the sense of \cite{uniform}. Indeed $L_N(q)$ appears as the special case $M=N$ of \cite[Example~3.1(5)]{uniform}.
In \cite[p.~195]{uniform} we defined  the incidence algebra of $L_N(q)$. Comparing that definition with Definitions \ref{def:Talg}, \ref{def:T} we find
that the incidence algebra of $L_N(q)$ is equal to $T$.
By \cite[Theorem~2.5]{uniform}, the standard module $V$ is a direct sum of irreducible $T$-modules. 
In Lemmas \ref{lem:Wfacts}, \ref{lem:Wfacts2} and Proposition \ref{prop:TDpair} we described the irreducible $T$-modules.
We now give a more detailed description using 
\cite[Theorem~2.5, Theorem~3.3(5)]{uniform} along with \cite[Section~2]{murali}. 
%We now interpret this description in terms of
% Leonard systems  \cite{ter2, LSnotes}.
Let $W$ denote an irreducible $T$-module, with endpoint $r$, dual endpoint $t$, and diameter $d$. Then
$0 \leq r \leq N/2$ and $t=r$ and $d=N-2r$.
 Moreover, the subspaces $E_iW$ and $E^*_iW$ have dimension one for $r \leq i \leq N-r$. 
By these comments
the sequence $(A; \lbrace E_i \rbrace_{i=r}^{N-r}; A^*; \lbrace E^*_i \rbrace_{i=r}^{N-r})$  acts on $W$ as a Leonard system in the sense of
\cite[Definition~1.4]{ter2}. For notational convenience, let $\Phi$ denote this Leonard system. By Lemma \ref{lem:AsTD} we see that $\Phi$  is bipartite
in the sense of \cite[Definition~5.1]{hanson}.
Using the bipartite feature and the form of the eigenvalues \eqref{eq:thi}, \eqref{eq:thsi} 
we find that $\Phi$  has
dual $q$-Krawtchouk type \cite[Example~20.7]{LSnotes} with parameters
\begin{align*}
&d(\Phi) = N-2r, \qquad \qquad h(\Phi) = \frac{q^{N-r}}{q-1}, \qquad \quad h^*(\Phi)=q^{-r}, \\
&s(\Phi) = -q^{2r-N-1}, 
\qquad \qquad \theta_0(\Phi) = q^r \lbrack N-2r \rbrack_q, \qquad \qquad \theta^*_0(\Phi) = q^{-r}.
\end{align*}
The Leonard system $\Phi$ is determined up to isomorphism by the above six parameters \cite[Proposition~9.8]{LSnotes}, and these parameters are
determined by $r, N, q$. By these comments  and \cite[Lemma~9.7]{cerzo},  a pair of irreducible $T$-modules are isomorphic if and only if they have the same endpoint.
For $0 \leq r \leq N/2$ let $\mu_r$ denote the multiplicity with which the irreducible $T$-module
with endpoint $r$ appears in the standard module $V$.  It follows from
 \cite[Theorem~3.3(5)]{uniform} that
 $\mu_0= 1$ and $\mu_r = \binom{N}{r}_q - \binom{N}{r-1}_q$ for $1 \leq r \leq N/2$.

\section{Acknowledgement} The author thanks Pierre-Antoine Bernard for a discussion about how the weighted adjacency matrix
in the present paper shows up in the context of symplectic dual polar graphs.

\bigskip

\noindent Paul Terwilliger \hfil\break
\noindent Department of Mathematics \hfil\break
\noindent University of Wisconsin \hfil\break
\noindent 480 Lincoln Drive \hfil\break
\noindent Madison, WI 53706-1388 USA \hfil\break
\noindent email: {\tt terwilli@math.wisc.edu }\hfil\break


\begin{thebibliography}{10}
%%%%%%%%%%%%%
\bibitem{bannai}
E.~Bannai, T.~Ito.
\newblock
{\em Algebraic Combinatorics, I. Association schemes.}
\newblock Benjamin/Cummings, Menlo Park, CA, 1984.



\bibitem{bbit}
E.~Bannai, E.~Bannai, T.~Ito, R.~Tanaka.
\newblock
{\em Algebraic Combinatorics.}
\newblock De Gruyter Series in Discrete Math and Applications 5.
De Gruyter, 2021.      \\
https://doi.org/10.1515/9783110630251
 
 
\bibitem{PAB}
P.~Bernard, N.~Cramp{\'e}, L.~Vinet.
\newblock
The Terwilliger algebra of symplectic dual polar graphs, the subspace lattices and $U_q(\mathfrak{sl}_2)$.
\newblock Preprint;
{\tt arXiv:2108.13819}.
 
 
%\bibitem{Biggs} N. Biggs {\it Algebraic Graph Theory. Second edition.}
 %     Cambridge University Press, Cambridge, 1993.
 
 
 
\bibitem{bcn}
A.~E.~Brouwer,  A.~Cohen, A.~Neumaier.
\newblock{\em Distance Regular-Graphs.}
\newblock Springer-Verlag, Berlin, 1989.


   
%   \bibitem{norton}
% P.~Cameron,  J.~Goethals, J.~Seidel.
% \newblock  The  Krein  condition,  spherical  designs,  Norton  algebras,  and permutation  groups.
% \newblock{\em  Indag.  Math.} 40  (1978)  196--206.

\bibitem{cameron}
P.~J.~Cameron.
\newblock  Projective and polar spaces.
\newblock QMW Maths Notes, 13. Queen Mary and Westfield College, School of Mathematical Sciences, London, 1992.



\bibitem{cerzo}
 D.~Cerzo.
 \newblock Structure of thin irreducible modules of a $Q$-polynomial distance-regular graph.
 \newblock{\em Linear Algebra Appl.}  433  (2010)  1573--1613; {\tt arXiv:1003.5368}.




\bibitem{dkt}   
E.~R.~van Dam, J.~H.~Koolen, H.~Tanaka.
\newblock Distance-regular graphs.
\newblock{\em Electron. J. Combin.} (2016) DS22.


\bibitem{delsarte}
P.~Delsarte.
\newblock
An algebraic approach to the association schemes of coding theory.
\newblock
{\em Philips Research Reports Suppl.} 10  (1973).







\bibitem{murali}
S.~Ghosh and M.~Srinivasan.
\newblock  
A $q$-analog of the adjacency matrix of the $n$-cube.
Preprint; 
{\tt arXiv:2204.05540}.





\bibitem{go} J. T.  Go.
 The Terwilliger algebra of the hypercube.
\newblock{\em
European J. Combin.}
{23} (2002) 399--429.


\bibitem{godsil}
C.~Godsil and J. Shawe-Taylor.
\newblock Distance-regularised graphs are distance-regular or distance-biregular.
\newblock{\em  J. Combin. Theory Ser. B}  43  (1987) 14--24.


\bibitem{hanson}		
E.~Hanson.
\newblock A characterization of bipartite Leonard pairs using the notion of a tail.
\newblock{\em  Linear Algebra Appl. } 452  (2014) 46--67; {\tt arXiv:1308.3826}.



\bibitem{itoOq}
T.~ Ito.
\newblock TD-pairs and the $q$-Onsager algebra.
\newblock{\em 
 Sugaku Expositions}  32  (2019)  205--232.
		



\bibitem{IKT}
T.~Ito, K.~Nomura,  P.~Terwilliger.
\newblock A classification of sharp tridiagonal pairs.
\newblock{\em  Linear Algebra Appl. } 435  (2011)  1857--1884; {\tt arXiv:1001.1812}.

	
	
	
\bibitem{itt} T. Ito, K. Tanabe,  P. Terwilliger.
 Some algebra related to $P$- and $Q$-polynomial association
schemes.
\newblock{\em
Codes and Association Schemes (Piscataway NJ, 1999), 167--192, DIMACS
Ser. Discrete Math. Theoret. Comput. Sci.}
{56}, Amer. Math. Soc., Providence RI 2001;
{\tt arXiv:math.CO/0406556}.	



\bibitem{tdanduq}
T.~Ito and P.~Terwilliger.
\newblock {Tridiagonal pairs and the quantum affine 
algebra
$U_q({\widehat{sl}}_2)$.}
\newblock {\em Ramanujan J.}
{13} (2007) 39--62;
{\tt arXiv:math.QA/0310042}.

\bibitem{qtet}
T.~Ito and P.~Terwilliger.
\newblock
 Distance-regular graphs and the $q $-tetrahedron algebra.
 \newblock {\em European J. Combin. } 30  (2009) 682--697; {\tt arXiv:math/0608694}.
 
\bibitem{augIto}
T.~Ito and P.~Terwilliger.
\newblock
The augmented tridiagonal algebra.
\newblock{\em  Kyushu J. Math.} 64 (2010) 8--144; {\tt arXiv:0904.2889}.



\bibitem{kim1}
J.~Kim.
\newblock Some matrices associated with the split decomposition for a $Q$-polynomial distance-regular graph.
\newblock {\em  European J. Combin. } 30  (2009)  96--113; {\tt arXiv:0710.4383}.
		

\bibitem{kim2}
J.~Kim.
\newblock A duality between pairs of split decompositions for a $Q$-polynomial distance-regular graph.
\newblock {\em  Discrete Math.}  310  (2010)  1828--1834; {\tt arXiv:0705.0167}.

		




\bibitem{leonard}
 D.~A.~Leonard.
 \newblock Orthogonal polynomials, duality and association schemes.
 \newblock{\em SIAM J. Math. Anal.} 13 (1982) 656--663.
 
 
 

\bibitem{nomTB}
K.~ Nomura and P.~Terwilliger.
\newblock Totally bipartite tridiagonal pairs.
 \newblock {\em Electron. J. Linear Algebra } 37  (2021) 434--491; {\tt arXiv:1711.00332}.
 
 
 \bibitem{pasc}
 A.~Pascasio.
 \newblock  On the multiplicities of the primitive idempotents of a $Q$-polynomial distance-regular graph.
\newblock{\em European J. Combin.}  23  (2002) 1073--1078.
 
 
 
 

\bibitem{murali2}
 M.~Srinivasan.
 \newblock  A positive combinatorial formula for the complexity of the $q$-analog of the $n$-cube.
\newblock{\em  Electron. J. Combin. } 19  (2012)  Paper 34, 14 pp.

 
 \bibitem{uniform}
 P.~Terwilliger.
 \newblock The incidence algebra of a uniform poset.
\newblock Coding theory and design theory, Part I, 
 193--212, IMA Vol. Math. Appl., 20, Springer, New York,  1990.
 
 
 
\bibitem{terwSub1} 
P.~Terwilliger. 
\newblock The subconstituent algebra of
an association scheme I.
\newblock{\em
J. Algebraic Combin.}
{ 1} (1992), 363--388.  

\bibitem{terwSub2} 
P.~Terwilliger. 
\newblock The subconstituent algebra of
an association scheme II.
\newblock{\em
J. Algebraic Combin.}
{ 2} (1993), 73--103.  

\bibitem{terwSub3} 
P.~Terwilliger. 
\newblock The subconstituent algebra of
an association scheme III. 
\newblock{ \em
J. Algebraic Combin.}
{ 2} (1993), 177--210.  

 
 
 
 
 

\bibitem{ter2}
P.~Terwilliger.
\newblock Two linear transformations each tridiagonal with respect to an eigenbasis of the other.
\newblock{\em  Linear Algebra Appl.} 330 (2001)  149--203; {\tt arXiv:math/0406555}.


 \bibitem{qSerre}
  P.~Terwilliger.
    \newblock Two relations that generalize the $q$-Serre relations and the
      Dolan-Grady relations. In
        \newblock {\em  Physics and
	  Combinatorics 1999 (Nagoya)}, 377--398, World Scientific Publishing,
	     River Edge, NJ, 2001;
	     {\tt arXiv:math.QA/0307016}.

\bibitem{introLP}
P.~Terwilliger.
\newblock Introduction to Leonard pairs.
\newblock 
Proceedings of the Sixth International Symposium on Orthogonal Polynomials, Special Functions and their Applications (Rome, 2001).
\newblock{\em J. Comput. Appl. Math. } 153  (2003)  463--475.



\bibitem{splitDec}
P.~Terwilliger.
\newblock The displacement and split decompositions for a $Q$-polynomial distance-regular graph.
\newblock{\em Graphs Combin. } 21  (2005)  263--276;
{\tt arXiv:math/0306142}.
		

\bibitem{aa}
P. Terwilliger.
\newblock  An algebraic approach to the Askey scheme of orthogonal polynomials.
 Orthogonal polynomials and special functions, 
 255--330, Lecture Notes in Math., 1883, Springer, Berlin,  2006; {\tt arXiv:math/0408390}. 



\bibitem{LSnotes}
P.~Terwilliger.
\newblock
Notes on the Leonard system classification.
\newblock{\em 
 Graphs Combin. } 37  (2021) 1687--1748;
{\tt arXiv:2003.09668}.

\bibitem{int}
P.~Terwilliger.
\newblock
Distance-regular graphs, the subconstituent algebra, 
and the $Q$-polynomial property.
Preprint; {\tt arXiv:2207.07747}.




\end{thebibliography}
\end{document}